\documentclass[11pt,a4paper]{article}
\usepackage{amsthm,amsfonts,amsmath}

\newtheorem{Theorem}{Theorem}
\newtheorem{Lemma}{Lemma}
\newtheorem{Proposition}{Proposition}
\newtheorem{Remark}{Remark}

\def\tr{{\rm\,trace\,}}

\title{Weak quasi contact metric manifolds and new characteristics of~K-contact and Sasakian manifolds}
\author{Vladimir Rovenski}

\begin{document}

\maketitle

\date{ }

\begin{abstract}
Quasi-contact metric manifolds (introduced by Y. Tashiro and then studied by several authors) are a natural extension of contact metric manifolds. Weak almost-contact metric manifolds, i.e., where the~linear complex structure on the contact distribution is replaced by a nonsingular skew-symmetric tensor, have been defined by the author and R. Wolak. In this paper, we study a weak analogue of quasi-contact metric manifolds.
Our main results generalize some well-known theorems and provide new criterions for K-contact and Sasakian manifolds
in terms of conditions on the curvature tensor and other geometric objects associated with the weak quasi-contact metric structure.


\end{abstract}


\section{Introduction}
\label{sec:00-ns}

Contact Riemannian geometry plays an important role in both mathematics and
physics.
It~considers a $(2n + 1)$-dimensional smooth manifold $M$ equipped with an \textit{almost-contact metric} (a.c.m.) structure $({f},\xi,\eta,g)$,
where $g$ is a Riemannian metric,
${f}$ is a
$(1,1)$-tensor, $\xi$ is a vector field and $\eta$ is a 1-form
{satisfying} 
\begin{align*}
 & {f}^2 = -{\rm id} +\eta\otimes\xi,\quad \eta(\xi)=1, \\
 & g({f} X,{f} Y)= g(X,Y) -\eta(X)\,\eta(Y),\quad \eta(X)=g(\xi,X),
\end{align*}
 where $X,Y\in\mathfrak{X}_M$ and
 $\mathfrak{X}_M$ are the Lie algebra of smooth vector fields on $M$.
{In}
 \cite{CG-1990}, D.~Chinea and C.~Gonzalez obtained a classification of a.c.m. manifolds
which was analogous to the classification of almost-Hermitian manifolds established by A. Gray and H.M.~Hervella; see~\cite{GH-80}.
Let~$M^{2n+1}({f},\xi,\eta,g)$ be an almost-contact metric manifold.
The tensor ${\cal N}^{\,(1)}$ is defined by
\[
 {\cal N}^{\,(1)} = [{f},{f}] + 2\,d\eta\otimes\,\xi,
\]
where
 $2\,d\eta(X,Y) = X(\eta(Y)) - Y(\eta(X)) - \eta([X,Y])
=(\nabla_X\eta)Y-(\nabla_Y\eta)X$,
and
\begin{equation*}
 [{f},{f}](X,Y) = {f}^2 [X,Y] + [{f} X, {f}Y] - {f}[{f}X, Y] - {f}[X, {f}Y] ,
\end{equation*}
is the~Nijenhuis torsion of ${f}$.
The (1,1)-tensor $h$ is defined by
\[
 h=(1/2)\,{\cal L}_\xi\,f,
\]
where $({\cal L}_Z\,f)X=[Z, fX] - f[Z,X]$ is the Lie derivative, {as in}  \cite{blair2010riemannian}.
On a contact metric manifold, $h$ vanishes if and only if $\xi$ is a Killing vector field.
A contact metric manifold $M(f,\xi,\eta,g)$ for which $\xi$ is a Killing vector field is called a K-{contact manifold}, as in \cite{blair2010riemannian}.
The following classes of a.c.m. manifolds are well {known:}


\begin{itemize}
\item[(1)]\hskip-1.5pt ${\cal M}$ -- normal a.c.m. manifolds, characterized by the equality ${\cal N}^{\,(1)}=0$.
\item[(2)]\hskip-1.5pt ${\cal S}$ -- Sasakian manifolds,
characterized by the equality; see \cite[Theorem~6.3]{blair2010riemannian},
\begin{equation}\label{E-S1}
 (\nabla_X{f})Y=g(X,Y)\,\xi -\eta(Y)X,\quad X,Y\in\mathfrak{X}_M .
\end{equation}
\item[(3)]\hskip-1.5pt ${\cal C}$ -- contact metric manifolds, characterized by the equality $d\eta{=}\Phi$, where $\Phi({X},{Y}){=}g({X},{f} {Y})$.
\item[(4)]\hskip-1.5pt ${\cal N}$ -- nearly-Sasakian manifolds, characterized by the equality;
see~\cite{blair1976}:
\begin{equation*}
 (\nabla_Y{f})Y = g(Y,Y)\,\xi -\eta(Y)Y,\quad Y\in\mathfrak{X}_M,
\end{equation*}
\item[(5)]\hskip-1.5pt ${\cal Q}$ -- quasi-contact metric (q.c.m.) manifolds, characterized by the equality
\begin{equation}\label{E-quasi-01}
 (\nabla_X{f})Y + (\nabla_{fX}{f})fY = 2\,g(X,Y)\,\xi
 -\eta(Y)\big(X + hX +\eta(X)\,\xi\big),\quad X,Y\in\mathfrak{X}_M.
\end{equation}

\end{itemize}

{Thus}
, ${\cal N}\subset{\cal S}={\cal M}\cap{\cal C}$
and ${\cal C}\subset{\cal Q}$.
For the last inclusion, see
\cite[Lemma~7.3]{blair2010riemannian}; the inverse is true for dimension 3, but for dimensions $>3$ it is an open question; see~\cite{KPSP-2020,CKPS-2016,KPSP-2014}.








\smallskip

By means of the almost-Hermitian cone (see Section~\ref{sec:01-ns}), a.c.m. manifolds and almost-Hermiti\-an manifolds correspond to each~other.

\begin{itemize}
\item[(1)] $M(f,\xi,\eta,\bar g)$ belongs to ${\cal M}$ if and only if $\bar M(J,\bar g)$ belongs to ${\cal H}$ -- Hermitian manifolds, defined by $[{J},{J}]=0$.
\item[(2)] $M(f,\xi,\eta,\bar g)$ belongs to ${\cal S}$ if and only if $\bar M(J,\bar g)$ belongs to ${\cal K}$ -- K\"{a}hler manifolds, defined by $\bar\nabla J=0$, where $\bar\nabla$ is the Levi--Civita connection for~$\bar g$.
\item[(3)] $M(f,\xi,\eta,\bar g)$ belongs to ${\cal C}$ if and only if $\bar M(J,\bar g)$ belongs to ${\cal AK}$ -- almost-K\"{a}hler manifolds, defined by $d\Omega=0$, where $\Omega(X,Y)=\bar g(X,JY)$.
\item[(4)] $M(f,\xi,\eta,\bar g)$ belongs to ${\cal N}$ if and only if $\bar M(J,\bar g)$ belongs to ${\cal NK}$ -- nearly-K\"{a}hler manifolds, defined by $(\bar\nabla_X J)X=0$.
\item[(5)] $M(f,\xi,\eta,\bar g)$ belongs to ${\cal Q}$ if and only if  $\bar M(J,\bar g)$ belongs to ${\cal QK}$ -- quasi-K\"{a}hler manifolds, defined by $(\bar\nabla_X J)Y+(\bar\nabla_{JX}J)JY = 0$.
\end{itemize}

Thus, ${\cal AK}\cap{\cal NK}={\cal K}={\cal QK}\cap{\cal H}$ and ${\cal AK},{\cal NK}\subset{\cal QK}$.

In \cite{RP-2,RWo-2,rov-117,rov-122,rov-128,rov-130}, we introduced and studied metric structures on a smooth manifold that generalized the a.c.m.
structures.
These so-called ``weak'' structures (where the~linear complex structure on the contact distribution is replaced by a nonsingular skew-symmetric tensor)
made it possible to take a new look at the classical theory and find new applications.

One may consider classes
$w{\cal M}$, $w{\cal S}, w{\cal C}, w{\cal N} and~w{\cal Q}$
of weak structures, defined similarly to the above classes,
${\cal M}$, ${\cal S}, {\cal C}, {\cal N} and {\cal Q}$.
%
Our previous works \cite{RP-2,RWo-2,rov-117,rov-122,rov-128,rov-130} are devoted to the classes $w{\cal M}$, $w{\cal S}, w{\cal C} and w{\cal N}$; see survey \cite{rov-survey24}.
This paper continues our study of the geometry of weak a.c.m. manifolds and discusses how the above weak a.c.m. structures relate to each other.
The above open question motivates us to study class $w{\cal Q}$ of
{weak q.c.m. manifolds}.
Our overall goal is to demonstrate that weak a.c.m. structures allow us to look at the theory of contact metric manifolds in a new~way.
To do this, we successfully extended classical theorems from contact geometry to the more general setting of weak q.c.m. manifolds
$w{\cal Q}$.


This~paper is organized as follows.
In~Section~\ref{sec:01-ns}, following the introductory Section~\ref{sec:00-ns}, we review the basics of weak
a.c.m. mani\-folds and prove Lemma~\ref{P-2.6}.
Section~\ref{sec:main} contains our main contributions -- Propositions~\ref{P-7.1}, \ref{P-2.3}, \ref{P-06}, \ref{P-05}
and five theorems -- where we generalize some well-known results and provide new criterions for K-contact manifolds (\mbox{Theorems~\ref{T-05}, \ref{T-04} and \ref{T-06}})
and Sasakian manifolds (Theorems~\ref{T-07} and \ref{T-08}).



\section{Preliminaries}
\label{sec:01-ns}



A {weak a.c.m. structure} on a smooth
manifold $M^{\,2n+1}$
is defined by a $(1,1)$-tensor ${f}$, a nonsingular $(1,1)$-tensor $Q$, a vector field $\xi$,
a 1-form $\eta$, and a Riemannian metric $g$
 such~that
\begin{align}\label{E-nS-2.2}
\nonumber
& {f}^2 = -Q +\eta\otimes\xi,\quad \eta(\xi)=1, \\
& g({f} X,{f} Y)= g(X,Q\,Y) -\eta(X)\,\eta(Y),\quad \eta(X)=g(\xi, X)\quad (X,Y\in\mathfrak{X}_M).
\end{align}
%
%
 {The} following~equalities are true for a.c.m. manifolds; see \cite[Proposition~1(a)]{RWo-2}:
\begin{align}\label{E-wquasi-1}
  {f}\,\xi=0,\quad \eta\circ{f}=0,\quad \eta\circ Q=\eta,\quad
  [Q, {f}]
 =
  [\widetilde{Q},{f}]
 =0,\quad
 \eta\circ\widetilde Q=0,\quad \widetilde{Q}\,\xi=0,
\end{align}
where $\widetilde Q= Q - {\rm id}_{\,TM}$ is a ``small''
tensor. According to the above, ${f}(\ker\eta)\subset\ker\eta$ and ${\rm rank}\,{f}=2\,n$.
In~this case, ${f}$ is skew-symmetric, and $Q$ is self-adjoint and positive definite.

A~weak a.c.m. structure satisfying \eqref{E-quasi-01} is called a {weak q.c.m. structure}.
A~weak a.c.m. structure satisfying $d\eta = \Phi$ is called a {weak contact metric structure}.
For a weak contact metric structure,
the 1-form $\eta$ is contact; see \cite{rov-128}.
A~1-form $\eta$ on
$M^{\,2n+1}$ is said to be \textit{contact} if $\eta\wedge (d\eta)^n\ne0$, e.g.,~\cite{blair2010riemannian}.
A~weak almost-contact structure is said to be {normal} if
${\cal N}^{\,(1)}=0$.
A~normal weak contact metric structure is called a {weak Sasakian structure}.
Recall that a weak a.c.m. structure is weak Sasakian if and only if it is a Sasakian structure; see \cite[Theorem~3]{RP-2}.

The following tensors on a.c.m. manifolds are well known; see~\cite{blair2010riemannian}:
\begin{eqnarray*}
 && {\cal N}^{\,(2)}(X,Y) = (\pounds_{{f} X}\,\eta)(Y) - (\pounds_{{f} Y}\,\eta)(X)
 = 2\,d\eta({f} X,Y) - 2\,d\eta({f} Y,X),  \\
 && {\cal N}^{\,(3)}(X) = (\pounds_{\xi}{f})X
 = [\xi, {f} X] - {f} [\xi, X],\\
 && {\cal N}^{\,(4)}(X) = (\pounds_{\xi}\,\eta)(X)
 = \xi(\eta(X)) - \eta([\xi, X]) = 2\,d\eta(\xi, X).
\end{eqnarray*}
 {For} a weak a.c.m. structure,
the following equality is true; see \cite{KPSP-2014} for $Q={\rm id}$:
\begin{align}\label{Eq-2.9}
 {\cal N}^{(2)}(X,Y) =
 (\nabla_{fX}\,\eta)(Y) - (\nabla_{Y}\,\eta)(fX)
 -(\nabla_{fY}\,\eta)(X) + (\nabla_{X}\,\eta)(fY).
\end{align}
 {Let} $M^{2n+1}({f},Q,\xi,\eta,g)$ be a weak a.c.m. manifold.
Define (1,1)-tensors $J$, ${\cal P}$ and a Riemannian metric
$\bar g$ on
the product
$\bar M = M^{2n+1}\times\mathbb{R}$
for $X,Y\in TM$ and $t\in\mathbb{R}$:
\begin{eqnarray*}
 && J(X,\, 0) = (fX ,\,-\eta(X)\partial_t),\ \
 J(0,\, \partial_t) = (\xi,\, 0),\ \
 {\cal P}(X,\, 0) = (QX,\, 0),\ \
 {\cal P}(0,\, \partial_t) = (0,\, \partial_t). \\
 && \bar g((X,0),(Y,0))=e^{-2t}g(X,Y),\quad \bar g((X,0), (0,\partial_t))=0,\quad \bar g((0,\partial_t),(0,\partial_t))=e^{-2t}.
\end{eqnarray*}
 {Then,} $J^{\,2}=-{\cal P}$ will hold and $\bar M(J, {\cal P}, \bar g)$ will be a weak almost-Hermitian manifold.
The tensors ${\cal N}^{\,(i)}\ (i=1,2,3,4)$ appear when we use
the integrability condition $[J, J]=0$ of $J$ to express the normality condition ${\cal N}^{\,(1)}=0$ of the weak a.c.m. structure;
see \cite{rov-117}.
For weak contact metric manifolds, we have ${\cal N}^{\,(2)}={\cal N}^{\,(4)}=0$ and
the trajectories of $\xi$ are geodesics, i.e., $\nabla_\xi\,\xi=0$; moreover,
${\cal N}^{\,(3)}=0$ if and only if $\,\xi$ is a Killing vector field; see~\cite{rov-117}.





The following result generalizes Proposition~2.6 in \cite{KPSP-2014}.

\begin{Lemma}\label{P-2.6}
For a weak q.c.m. structure $(f,Q,\xi,\eta,g)$,
the following equalities are true:
\begin{align}
\label{E-C2}
 & (\nabla_X\,\eta)(QY) + (\nabla_{fX}\,\eta)(fY) + 2\,g(fX,Y) = 0,\\
 \label{E-C3}
 & \nabla_\xi\,f = 0,\\
 \label{E-C4}
  & \nabla_\xi\,\xi = 0,\quad \nabla_\xi\,\eta = 0,\\
 \label{E-Qnabla}
 & Q\,\nabla\,\xi = (\nabla\,\xi)Q = - f - fh, \\
\label{E-nS-xi}
& \pounds_{\xi}\,{Q} = \nabla_{\xi}\,{Q} = 0, \\
 \label{E-31A}
 & hf + f h =
  0 ,\\
\label{E-31B}
 & h\,Q - Q\,h = 0.
 \end{align}
\end{Lemma}

\begin{proof}
Setting $X=\xi$ in \eqref{E-quasi-01} and using $f\xi=0$ and $h\xi=0$, we obtain \eqref{E-C3}.
The equality \eqref{E-C4} follows directly from~\eqref{E-C3}.
From \eqref{E-quasi-01} with $Y=\xi$, we obtain
 $f\nabla_X\,\xi = X - \eta(X)\,\xi + hX$.
Multiplying this by $f$ and using $f^2=-Q+\eta\otimes\xi$, see \eqref{E-nS-2.2}, we obtain
$Q\,\nabla\,\xi
= - f - fh$;
see \eqref{E-Qnabla}.

Replacing $X$ and $Y$ with $fX$ and $fY$, respectively, in \eqref{E-quasi-01} gives
\[
(\nabla_{fX} f)fY + (\nabla_{f^2X} f)f^2Y = 2\,g(fX,fY)\,\xi,
\]
and then using ${f}^2 = -Q +\eta\otimes\xi$, see \eqref{E-nS-2.2}, we obtain
\begin{align*}
& (\nabla_{QX} f)(QY) + (\nabla_{fX} f)fY
= 2\,g(QX, Y)\,\xi -2\eta(X)\,\eta(Y)\,\xi\\
& + \eta(Y)(\nabla_{QX} f)\xi + \eta(X)(\nabla_{\xi} f)QY
- \eta(X)\,\eta(Y)(\nabla_{\xi} f)\xi .
\end{align*}
 {Subtracting} from this \eqref{E-quasi-01}, we obtain
\begin{align}\label{Eq-2.36}
\notag
& (\nabla_{QX} f)(QY) - (\nabla_X f)Y + 2\,g(\widetilde QX, Y)\,\xi \\
\notag
& = \eta(Y)(\nabla_{QX} f)\xi + \eta(X)(\nabla_{\xi} f)QY
- \eta(X)\,\eta(Y)(\nabla_{\xi} f)\xi \\
& -\eta(X)\,\eta(Y)\,\xi +\eta(Y)X + \eta(Y) hX .
\end{align}
 {From} \eqref{E-C3} and the definition of $h$, we obtain
\begin{align}\label{Eq-2.11b}
 2\,hX= f\nabla_X\,\xi - \nabla_{fX}\,\xi .
\end{align}
 {Thus,} from \eqref{Eq-2.36} and \eqref{E-C3} we obtain
\begin{align*}
& (\nabla_{QX} f)(QY) - (\nabla_X f)Y + 2\,g(\widetilde QX, Y)\,\xi \\
& =\eta(Y) \big( X - f\nabla_{QX}\,\xi
 -\eta(X)\,\xi  + (1/2)f\nabla_X\,\xi - (1/2)\nabla_{fX}\,\xi\big) .
\end{align*}
 {Setting} $Y=\xi$ in the above equation, we obtain
\begin{align*}
 2X - 2\eta(X)\,\xi
 - \nabla_{fX}\,\xi
 = f\nabla_{X}\,\xi .
\end{align*}
 {Applying} $f$ to the above equation and using \eqref{E-nS-2.2} and $\eta(\nabla_{X}\,\xi)=0$,
we obtain
\begin{align*}
 2fX = f\nabla_{fX}\,\xi + f^2\nabla_{X}\,\xi
 = f\nabla_{fX}\,\xi - Q\nabla_{X}\,\xi .
\end{align*}
 {Multiplying} this by $Y$, we obtain \eqref{E-C2}.

{Using} \eqref{Eq-2.11b} and $f^2=-Q+\eta\otimes\xi$, see \eqref{E-nS-2.2}, we find
\begin{align*}
 2(hf+ f h)X = f^2\nabla_X\,\xi-\nabla_{f^2X}\,\xi
 = \nabla_{\widetilde Q}\,\xi - \widetilde Q\nabla_X\,\xi
 = (\pounds_{\xi}\,{\widetilde Q}) X.
\end{align*}
 {Using} \eqref{E-C3} and \eqref{E-C4},
we find $\nabla_{\xi}\,{\widetilde Q} = 0$.
We can also calculate the following, using \eqref{E-C3} and \eqref{E-Qnabla}:
\begin{align*}
 & 2hX=(\pounds_{\xi}{\widetilde Q})X
 =f\nabla_X\,\xi - \nabla_{fX}\,\xi \\
 & =-Q^{-1}f^2 -Q^{-1}f^2h+Q^{-1}f^2
 -Q^{-1}f(fh+(1/2)\pounds_{\xi}{\widetilde Q})X \\
 & = 2hX -(1/2)Q^{-1}f(\pounds_{\xi}{\widetilde Q})X.
\end{align*}
 {Thus,} $\pounds_{\xi}{\widetilde Q}=0$, which proves
\eqref{E-nS-xi} and \eqref{E-31A}.
Multiplying \eqref{E-31A} by $f$ and using \eqref{E-31A} gives~\eqref{E-31B}.
Using the above, we can complete the proof of \eqref{E-Qnabla}: $(\nabla\,\xi)QX
=Q^{-1}(-fQX-fhQX) = - fX - fhX$.
\end{proof}

\begin{Remark}\rm
For weak contact metric manifolds, we generally have
$\nabla_\xi\,f\ne0$;
see \cite[Corollary~1]{RP-2}, which differs from \eqref{E-C3}.
Thus, the class $w{\cal C}$ is not contained in $w{\cal Q}$, although ${\cal C}\subset{\cal Q}$.
\end{Remark}

\section{Main Results}
\label{sec:main}

The next theorem completes \cite[Theorem~2]{rov-117} and characterizes K-contact manifolds among weak q.c.m. manifolds satisfying \eqref{E-nS-xi}
by using the following property of K-contact manifolds; see~\cite{blair2010riemannian}:
\begin{equation}\label{E-30}
 \nabla\,\xi = -f .
\end{equation}

\begin{Theorem}\label{T-05}
Let $M(f,Q,\xi,\eta,g)$ be a weak q.c.m. manifold and \eqref{E-30} be valid.
Then, $Q={\rm id}$, and $M(f,\xi,\eta,g)$ is a K-contact manifold.
\end{Theorem}

\begin{proof}
{Let a weak q.c.m. manifold $M(f,Q,\xi,\eta,g)$
satisfy the condition
\eqref{E-30}.
From \eqref{E-Qnabla}~and~\eqref{E-30}}, we obtain
$f(h-\widetilde Q)=0$. Since $f$ is non-degenerate on $\ker\eta$ and $h\xi=\widetilde Q \xi=0$ is true, we find $h=\widetilde Q$.
On the other hand, using
\eqref{E-nS-xi}, we obtain $\pounds_{\xi}{\widetilde Q}=0$. Therefore, from \eqref{E-31A} and $[Q,f]=0$, see \eqref{E-wquasi-1}, we find $f\widetilde Q=0$. From this, since $f$ is non-degenerate on $\ker\eta$, we obtain $\widetilde Q=0$. Therefore, $M(f,\xi,\eta,g)$ is a q.c.m. manifold.
Since $h=0$ also holds, according to \cite[Theorem~3.2]{KPSP-2014}, $M(f,\xi,\eta,g)$ is a contact metric manifold.
Next,
using \eqref{E-30} and the skew-symmetry of $f$, we can conclude that $\xi$ is a Killing vector field:
\begin{align*}
 g(\nabla_X\xi,Y)+g(\nabla_Y\xi,X)=-g(fX,Y)-g(fY,X)=0.
 \end{align*}
  {Therefore,} $M(f,\xi,\eta,g)$ is a K-contact manifold.
\end{proof}

The
curvature tensor is given by
$R_{X,Y}\,Z=\nabla_X\nabla_Y Z-\nabla_Y\nabla_X Z
-\nabla_{[X,Y]} Z$.

The~following result generalizes \cite[Proposition~7.1]{blair2010riemannian} on contact metric manifolds.

\begin{Proposition}\label{P-7.1}
For a weak q.c.m. structure $(f,Q,\xi,\eta,g)$, 
the following equalities are true:
\begin{align}\label{E-Lie-Q1}
 & (\nabla_\xi\,h)X = Q^{-1}\big( fX - h^2fX\big) -f R_{X,\xi}\,\xi, \\
 \label{E-Lie-Q2}
 & Q R_{\,\xi,X}\,\xi - f R_{\,\xi, fX}\,\xi
 = 2 h^2 X + (Q+Q^{-1}) f^2 X .
\end{align}
\end{Proposition}

\begin{proof}
Using \eqref{E-Qnabla} and \eqref{E-C3}, we can compute
\begin{align*}
 QR_{\,\xi,X}\,\xi =
 f\nabla_\xi(-X - hX) + f[\xi,X] + fh[\xi,X].
\end{align*}
 {Applying} $f$ and using \eqref{E-nS-2.2}, \eqref{E-31B}
and \eqref{E-Qnabla}, we obtain
\begin{align*}
 Qf R_{\,\xi,X}\,\xi & =
 f^2\nabla_\xi(-X - hX) + f^2[\xi,X] + f^2h[\xi,X] \\
& = Q\nabla_\xi(X + hX) - \eta(\nabla_\xi(X + hX))\xi
-Q[\xi,X] +\eta([\xi,X])\xi -Qh[\xi,X] \\
& = Q\big( (\nabla_\xi h)X + \nabla_X\,\xi + h\nabla_X\,\xi \big).
\end{align*}
 {According} to the above, since $Q$ is nonsingular, the following is true:
\begin{align*}
 f R_{\,\xi,X}\,\xi =
 (\nabla_\xi h)X + \nabla_X\,\xi + h\nabla_X\,\xi.
\end{align*}
 {Using} \eqref{E-Qnabla}{,}
 \eqref{E-31A} {and} \eqref{E-31B}, we obtain the following formula \eqref{E-Lie-Q1}:
\begin{align}\label{E-fR}
 f R_{\,\xi,X}\,\xi = (\nabla_\xi\,h)X
 + Q^{-1}\big(- f X + h^2 f X \big).
\end{align}
 {Next,} applying $f$ to \eqref{E-Lie-Q1} and using
$f^2X=-QX+\eta(X)\,\xi$, see \eqref{E-nS-2.2}, we find
\begin{align}\label{E-fR2}
 Q R_{\,\xi,X}\,\xi
 = h^2 X + Qf^2 X -f(\nabla_\xi\,h)X .
\end{align}
 {Using} \eqref{E-C3} and \eqref{E-31A}, we obtain $(\nabla_\xi\,h)fX=-f(\nabla_\xi\,h)X$.
From \eqref{E-fR}, using $h^2f=fh^2$, see \eqref{E-31B},
we obtain
\begin{align}\label{E-fR3}
 f R_{\,\xi, fX}\,\xi = (\nabla_\xi\,h)fX
 + Q^{-1}( - f^2 X + h^2 f^2 X)
 = -h^2X - Q^{-1}f^2X - f(\nabla_\xi\,h)X .
\end{align}
 {Subtracting} \eqref{E-fR3} from \eqref{E-fR2} yields \eqref{E-Lie-Q2}.
\end{proof}




The following result generalizes Proposition~2.3 in \cite{KPSP-2014} for $Q={\rm id}$.

\begin{Proposition}\label{P-2.3}
 Let a weak a.c.m. structure $(f,Q,\xi,\eta,g)$ satisfy
 $\nabla_\xi\,f = 0$; see \eqref{E-C3}.
 Then,
 \begin{align}\label{Eq-2.14}
 g(hX,Y) - g(hY,X) = -(1/2)\,{\cal N}^{(2)}(X,Y);
\end{align}
hence,  $h$ is self-adjoint if and only if
${\cal N}^{(2)}\equiv0$.
\end{Proposition}

\begin{proof} Note that $h\xi=2{\cal N}^{\,(3)}(\xi)=0$ and
$hX\perp\xi$ for $X\perp\xi$.
Using $\nabla_\xi\,f = 0$, we obtain
\begin{align}\label{Eq-2.11}
 2hX =
 - \nabla_{fX}\,\xi + f\nabla_{X}\,\xi.
\end{align}
 {From} \eqref{Eq-2.9} and \eqref{Eq-2.11}, we obtain \eqref{Eq-2.14}, which completes the proof.
\end{proof}

The following result generalizes Theorem~4.2 in \cite{CKPS-2016}.

\begin{Proposition}\label{P-06}
Let $M(f,Q,\xi,\eta,g)$ be a weak q.c.m. manifold.
If $\xi$ is a Killing vector field,
then the tensor $h$ is skew-symmetric
(and $h^2$ is nonpositive definite); if the contact distribution $\ker\eta$ is integrable,
then $h$ is self-adjoint (and $h^2$ is nonnegative definite).
\end{Proposition}

\begin{proof}
Using \eqref{E-Qnabla} and \eqref{E-31A}--\eqref{E-31B}, we can calculate
\begin{align*}
 & g(\nabla_{X}\,\xi, Y) \pm g(\nabla_{Y}\,\xi, X)
 = g(Q^{-1}(-f-fh)X, Y) \pm g(Q^{-1}(-f-fh)Y, X) \\
 & = -g(Q^{-1} f(h\pm h^*)X, Y) = g((h\pm h^*)X, Q^{-1} fY),
\end{align*}
where $h^*$ is the conjugate tensor to $h$.
According to the above, if $\xi$ is a Killing vector field,
that is, $g(\nabla_{X}\,\xi, Y) + g(\nabla_{Y}\,\xi, X) =0$
for all $X,Y$, then $h+h^*=0$,
and if the distribution $\ker\eta$ is integrable then $g([Y,X],\xi)=g(\nabla_{X}\,\xi, Y) - g(\nabla_{Y}\,\xi, X)=0$ 
\end{proof}

For a weak a.c.m. manifold $M(f,Q,\xi,\eta,g)$,
we will build an $f$-\textit{basis}
consisting of mutually orthogonal nonzero vectors
at a point $x\in M$.
Let $e_1\in(\ker\eta)_x$ be a unit eigenvector of the self-adjoint operator $Q>0$ with the eigenvalue $\lambda_1>0$. Then, $fe_1\in(\ker\eta)_x$ is orthogonal to $e_1$
and $Q(fe_1) = f(Qe_1) = \lambda_1 fe_1$.
Thus, the subspace orthogonal to the plane $span\{e_1,fe_1\}$ is $Q$-invariant.
 There exists a unit vector $e_2\in(\ker\eta)_x$ such that $e_2\perp span\{e_1, fe_1\}$
and $Q\,e_2= \lambda_2 e_2$  for some $\lambda_2>0$.
Obviously, $Q(fe_2) = f(Q\,e_2) = \lambda_2 fe_2$.
All five vectors $\{\xi, e_1, fe_1,e_2, fe_2\}$ are nonzero and mutually orthogonal.
Continuing in the same manner, we find a basis $\{\xi, e_1, fe_1,\ldots, e_n, fe_n\}$ of $T_x M$
consisting of mutually orthogonal vectors; see \cite{rov-128}.
Note that $g(fe_i,fe_i)=g(Qe_i,e_i)=\lambda_i$
and $\tr Q=2\sum_{i=1}^n \lambda_i$.

\smallskip


The following condition is stronger than $\nabla_\xi\,Q=0$, see \eqref{E-nS-xi}, and is valid when $Q={\rm id}$:
\begin{align}\label{E-nS-10}
 (\nabla_X\,Q)\,Y=0\quad (X,Y\in\ker\eta).
\end{align}
 {The} exterior derivative of $\Phi$ is given by
\begin{eqnarray*}
 d\Phi(X,Y,Z) && = ({1}/{3})\,\big\{ X\,\Phi(Y,Z) + Y\,\Phi(Z,X) + Z\,\Phi(X,Y) \notag\\
 && -\,\Phi([X,Y],Z) - \Phi([Z,X],Y) - \Phi([Y,Z],X)\big\}.
\end{eqnarray*}
 {The} next result completes Theorem~3.2 in \cite{KPSP-2014}.

\begin{Proposition}\label{P-05}
Let a weak q.c.m. manifold
satisfy the condition \eqref{E-nS-10}.
If $h$ is self-adjoint, then $\eta$ is a contact form and $d\Phi=0$.
\end{Proposition}

\begin{proof}
According to Lemma~\ref{P-2.6}, conditions \eqref{E-C2}--\eqref{E-C4} are true.
From \eqref{E-C2}, we obtain
\begin{align}\label{Eq-3.3}
  (\nabla_X\,\eta)(QY) - (\nabla_Y\,\eta)(QX)
  + (\nabla_{fX}\,\eta)(fY) - (\nabla_{fY}\,\eta)(fX)
  = - 4\,g(fX,Y) .
\end{align}
 {According} to \eqref{E-C3} and Proposition~\ref{P-2.3}, \eqref{Eq-2.14} is true;
thus, ${\cal N}^{(2)}\equiv0$ holds by symmetry of $h$.
Using \eqref{Eq-2.9}, we~obtain
\begin{align}\label{Eq-3.4}
  (\nabla_{fX}\,\eta)(Y) - (\nabla_{Y}\,\eta)(fX)
  -(\nabla_{fY}\,\eta)(X) + (\nabla_{X}\,\eta)(fY) = 0 .
\end{align}
{Replacing} $X$ with $fX$ in \eqref{Eq-3.4} and using
$\nabla_{\xi}\,\eta=0$ (since $\nabla_\xi\,\xi=0$;
see \eqref{E-C4}) and $(\nabla_{Y}\,\eta)(\xi)=0$, we~obtain
\begin{align}\label{Eq-3.5}
   (\nabla_{Y}\,\eta)(QX) -(\nabla_{QX}\,\eta)(Y)
  - (\nabla_{fY}\,\eta)(fX) + (\nabla_{fX}\,\eta)(fY) = 0 .
\end{align}
 {Subtracting} \eqref{Eq-3.5} from \eqref{Eq-3.3}
gives
\begin{align*}
 2(\nabla_{X}\,\eta)(Y) - 2(\nabla_{Y}\,\eta)(X)
 +(\nabla_{X}\,\eta)(\widetilde QY)
 -2(\nabla_{Y}\,\eta)(\widetilde QX)
 +(\nabla_{\widetilde QX}\,\eta)(Y) = 4\,\Phi(X,Y).
\end{align*}
 {Using} \eqref{E-nS-10}, we find
$(\nabla_{X}\,\eta)(\widetilde QY)=(\nabla_{Y}\,\eta)(\widetilde QX)=0$;
therefore,
\begin{align*}
 (\nabla_{X}\,\eta)(Y) - (\nabla_{Y}\,\eta)(X)
 +(1/2)\,(\nabla_{\widetilde QX}\,\eta)(Y) = 2\,\Phi(X,Y) .
\end{align*}
 {Hence,} using
$(\nabla_{X}\,\eta)(Y) - (\nabla_{Y}\,\eta)(X)=2d\eta(X,Y)$
and $(\nabla_{\widetilde QX}\,\eta)(Y)=2d\eta(\widetilde QX,Y)$,
we obtain
\begin{align*}
 d\eta(X+\frac12\,\widetilde QX, Y) = \Phi(X,Y).
\end{align*}
 {In} particular, $d\Phi=0$.
In terms of the $f$-basis, we obtain
$\frac12(\lambda_i+1)\,d\eta(e_i,Y)=\Phi(e_i,Y)$
and
$\Phi(e_i,e_j)=0$, $\Phi(e_i,fe_j)=-\lambda_i \delta_{ij}$,
where $\lambda_i>0$.
Therefore,
\[
 \eta\wedge (d\eta)^n(\xi, e_1, fe_1,\ldots, e_n, fe_n)
 =(d\eta)^n(e_1, fe_1,\ldots, e_n, fe_n) \ne0,
\]
and $\eta$ is a contact form.
\end{proof}




\begin{Theorem}\label{T-07}
If a weak q.c.m. manifold $M(f,Q,\xi,\eta,g)$ satisfies the condition \eqref{E-S1},
then $Q={\rm id}$ and
$M(f,\xi,\eta,g)$ is a Sasakian manifold.
\end{Theorem}

\begin{proof}
Using \eqref{E-S1} in \eqref{E-quasi-01}, we obtain the following for any $X,Y\in TM$:
\begin{align*}
 g(\widetilde Q X,Y)\,\xi + \eta(Y) hX =0.
\end{align*}
 {Since} $hX\perp\xi$, we obtain $\widetilde Q=0$ and $h=0$; therefore, $M(f,\xi,\eta,g)$ is a quasi-contact metric manifold.
According to \cite[Theorem~2.2]{KPSP-2020} ($h$ is self-adjoint),  $M(f,\xi,\eta,g)$ is a contact metric manifold.
According to the~conditions and \cite[Theorem~6.3]{blair2010riemannian}, $M(f,\xi,\eta,g)$ is a Sasakian manifold.
\end{proof}

On a contact metric manifold $M^{2n+1}$, we have $Ric (\xi,\xi)=2n - \tr h^2$;
see \cite[Corollary~7.1]{blair2010riemannian}.
The~following theorem generalizes this property
and generalizes \cite[Theorem~A]{KPSP-2020}.

\begin{Theorem}\label{T-04}
Let $M(f,Q,\xi,\eta,g)$ be a weak q.c.m. manifold such that
$K(\xi,X) + K(\xi, fX)\ge0$ for all nonzero $X\perp\xi$.
Then,
\begin{align}\label{E-lambda-1}
 \max_{1\le i\le n}\{\lambda_i\}\cdot Ric (\xi,\xi)\ge n -\tr h^2 +\frac1{4n}(\tr Q - 1)^2 .
\end{align}
 {If} $\tr\widetilde Q=0$ and the equality in \eqref{E-lambda-1} holds,
then $\widetilde Q=0$, $\lambda_i=1\ (1\le i\le n)$
and $Ric (\xi,\xi)=2n-\tr h^2$;
moreover, if $\xi$ is a Killing vector field, then
$M(f,\xi,\eta,g)$ is
a K-contact manifold.
\end{Theorem}

\begin{proof}
From \eqref{E-Lie-Q2} with $X=e_i$ and then $X=fe_i$, we can obtain
formulas with sectional curvature,
\begin{align*}
 & \lambda_i \big( K(\xi,e_i) + K(\xi, fe_i) \big)
 =  \lambda_i^2 + 1 -2 g(h^2 e_i, e_i) ,\\
 & \lambda_i\big(K(\xi,fe_i) + K(\xi, e_i) \big)
 = \lambda_i^2 + 1 - 2 g(h^2 fe_i/\|fe_i\|, fe_i/\|fe_i\|) \big) ,
 \end{align*}
using an $f$-basis $\{\xi,e_1,\ldots,e_n, fe_1,\ldots, fe_n\}$. From the above, we obtain
\begin{align}\label{E-K2}
 \sum\nolimits_{i=1}^n \lambda_i \big( K(\xi,e_i) + K(\xi, fe_i)
 \big) = n  - \tr h^2 + \sum\nolimits_{i=1}^n \lambda_i^2.
\end{align}
{Using} the equality $\tr Q=1+2\sum_{i=1}^n \lambda_i$ and the well-known inequality
$\sum_{i=1}^n \lambda_i^2 \ge \frac1n(\sum_{i=1}^n \lambda_i)^2$,
from \eqref{E-K2}, we obtain \eqref{E-lambda-1}.
If the equality in \eqref{E-lambda-1}
is true, then
$\lambda_1=\ldots=\lambda_n$.
Note that $\tr Q =2n+1+\tr\widetilde Q$.
According to the condition $\tr\widetilde Q=0$, we can obtain
$\lambda_i=1$ and $\widetilde Q=0$;
thus, $M(f,\xi,\eta,g)$ is a q.c.m. manifold
and $Ric (\xi,\xi)=2n-\tr h^2$.
If $\xi$ is a Killing vector field, then according to \cite[Theorem~A]{KPSP-2020}, $M(f,\xi,\eta,g)$ is a K-contact metric manifold.
\end{proof}



It is well known that a K-contact structure satisfies the following condition:
\begin{equation}\label{E-S3}
 R_{X,\xi}\,\xi = -X - \eta(X)\,\xi,
\end{equation}
thus, the sectional curvature of all planes containing $\xi$ is equal
to 1.

The following result complements \cite[Theorem~7.2]{blair2010riemannian}
and \cite[Theorem~A]{KPSP-2020} on K-contact mani\-folds.

\begin{Theorem}\label{T-06}
Let a weak q.c.m. manifold $M(f,Q,\xi,\eta,g)$ satisfy
\eqref{E-S3}.
If $\xi$ is a Killing vector field,
then, $Q={\rm id}$ and $M(f,\xi,\eta,g)$  is a K-contact manifold.
\end{Theorem}

\begin{proof}
According to \eqref{E-S3}, we obtain $K_{\xi,X}=1$ for nonzero vectors $X\perp\xi$.
Using this in \eqref{E-Lie-Q2}, we obtain
\begin{align*}
 Q(-X) - f(-fX) = 2h^2X +(Q+Q^{-1}) f^2X.
\end{align*}
 {Using} \eqref{E-nS-2.2} and $Q={\rm id}+\widetilde Q$, we simplify the above equation with $X\perp\xi$ to the following:
\begin{align*}
  2h^2X & = f^2X -QX -(Q+Q^{-1}) f^2X \\
 & = -QX +\eta(X)\,\xi -QX +(Q+Q^{-1})QX - \eta(X)(Q+Q^{-1})\xi \\
 & = -2QX + Q^2X +X
 = -2\widetilde QX + Q^2X -X
 \\
 & = -2\widetilde QX +\widetilde Q^2X +2\widetilde QX
  = \widetilde Q^2X.
\end{align*}
  {Therefore,} $2h^2 = \widetilde Q^2\ge0$, and consequently, $\tr h^2\ge0$.
 On the other hand, by Proposition~\ref{P-06} and the conditions,
 $h$ is skew-symmetric and $h^2\le0$. Therefore, $\widetilde Q=0$,
and we conclude that
$M(f,\xi,\eta,g)$
is a q.c.m. manifold
with a Killing vector field $\xi$.
Hence, according to \cite[Theorem~A]{KPSP-2020},
$M(f,\xi,\eta,g)$  is a K-contact manifold.
\end{proof}

It is well known that a contact metric structure is Sasakian if and only if
\begin{equation}\label{E-S2}
 R_{X,Y}\,\xi = \eta(Y)X - \eta(X)Y.
\end{equation}

The following result complements \cite[Proposition~7.6]{blair2010riemannian} on contact metric manifolds
and \cite[Theorem~4.3]{CKPS-2016} on q.c.m. manifolds.

\begin{Theorem}\label{T-08}
Let a weak q.c.m. manifold $M(f,Q,\xi,\eta,g)$ satisfy
conditions
$\tr h^2\le0$ and \eqref{E-S2}.
Then, $Q={\rm id}$ and $M(f,\xi,\eta,g)$  is a Sasakian manifold.
\end{Theorem}

\begin{proof}
According to \eqref{E-S2}, we can obtain $R_{\xi,X}\,\xi=-X$ for $X\perp\xi$.
As in the proof of Theorem~\ref{T-06}, we conclude that
$2h^2=\widetilde Q^2\ge0$. From this and the conditions, we obtain $\tr \widetilde Q^2=0$;
hence, $\widetilde Q=0$ and $M(f,\xi,\eta,g)$ is a q.c.m. manifold with condition \eqref{E-S2}.
Hence, according to \cite[Theorem~4.3]{CKPS-2016},
$M(f,\xi,\eta,g)$  is a Sasakian manifold.
\end{proof}




\section{Conclusions}

This paper contains substantial new mathematics that successfully
extends important concepts and theorems about contact Riemannian manifolds for the case of weak q.c.m. manifolds and provides new tools for studying K-contact and Sasakian structures.


\baselineskip=12.8pt

\begin{thebibliography}{99}

\bibitem{KPSP-2020}
Bae, J., Park, J.H. and Sekigawa, K.
Quasi contact metric manifolds with Killing characteristic vector fields. Bull. Korean Math. Soc. 57 (2020), No. 5, 1299--1306.

\bibitem{blair2010riemannian}
 Blair, D.E. \emph{Riemannian geometry of contact and symplectic manifolds},
 Second edition,  Springer-Verlag, New York, 2010.

\bibitem{blair1976}
 Blair, D.E., Showers, D.K. and Komatu, Y. Nearly Sasakian manifolds. {Kodai Math. Sem. Rep.} {1976}, {27}, 175--180.

\bibitem{CKPS-2016}
 Chai, Y.D., Kim, J.H., Park, J.H., Sekigawa K., Shin, W.M.
Notes on quasi contact metric manifolds.
An. \c{S}tiin\c{t}, Univ. Al. I. Cuza Ia\c{s}i Mat. (N.S.)
Tomul LXII, 2016, f. 2, vol. 1, 349--359.

\bibitem{CG-1990}
Chinea, D. and Gonzalez, C. A classification of almost contact metric manifolds. Ann. Mat. Pura Appl. (IV) 156 (1990), 15--36.



\bibitem{GH-80}
 Gray, A. and Hervella, L.M. The sixteen classes of almost Hermitian manifolds and their linear imvarients, Ann. Mat. Pura Appl. 123 (1980), 35--58.



\bibitem{KPSP-2014}
Kim, J.H., Park, J.H. and Sekigawa, K. A generalization of contact metric manifolds, Balkan J. Geom. Appl., 19 (2014), 94--105.

\bibitem{RP-2}
 Patra, D.S. and Rovenski, V. On the rigidity of the Sasakian structure and characterization of cosymplectic manifolds.
Differential Geometry and its Applications, 90 (2023) 102043.

\bibitem{RWo-2}
 Rovenski, V. and Wolak, R. {New metric structures on $\mathfrak{g}$-foliations}, Indagationes Mathema\-ticae, 33 (2022), 518--532.

\bibitem{rov-117}
 Rovenski, V. Generalized Ricci solitons and Einstein metrics on weak K-contact manifolds.
Communications in Analysis and Mechanics, 2023, Volume 15, Issue 2: 177--188.



\bibitem{rov-122}
Rovenski, V. {Weak nearly Sasakian and weak nearly cosymplectic manifolds}. Mathematics, 2023, {11}\,(20), 4377.

\bibitem{rov-128}
Rovenski, V. On the splitting of weak nearly cosymplectic manifolds, Differential Geometry and its Applications, 94 (2024) 102142.

\bibitem{rov-130}
Rovenski, V. {Characterization of Sasakian manifolds},
Asian-European Journal of Mathematics, {17}, No. 03, 2450030 (2024) (15 pages).


\bibitem{rov-survey24}
Rovenski, V. Weak almost contact structures: a survey. Preprint, 17 p. 2024. arXiv:2408.13827.



\end{thebibliography}
\end{document}